\documentclass{amsart}
\usepackage{amsmath,amsfonts,amsthm,enumerate,amssymb,textcomp}
\usepackage[mathscr]{euscript}
\usepackage[cp1250]{inputenc}
\usepackage[T1]{fontenc}
\def\={\discretionary{-}{-}{-}}

\newtheorem{theorem}{Theorem}
\newtheorem{lm}[theorem]{Lemma}
\newtheorem{wn}[theorem]{Corollary}
\theoremstyle{remark}
\newtheorem{definition}[theorem]{Definition}
\newtheorem{remark}[theorem]{Remark}

\author{Valentin A. Skvortsov}
\address{Moscow Center for Fundamental and Applied Mathematics\endgraf and \endgraf Mathematical Department\endgraf Moscow State University \endgraf Moscow 119991 \endgraf Russia\endgraf  e-mail: {\tt vaskvor2000@yahoo.com}}
\author{Piotr Sworowski}
\address{Casimirus the Great University\endgraf Institute of Mathematics\endgraf Powsta\'nc\'ow Wielkopolskich 2\endgraf 85-090 Bydgoszcz\endgraf Poland\endgraf e-mail: {\tt piotrus@ukw.edu.pl}}
\title{On the relation between Denjoy--Khintchine and HK$_{r}$-integrals}
\begin{document}
\begin{abstract}
We locate Musial\,\&\,Sagher's concept of HK$_r$-inte\-gration within the approximate Henstock--Kurzweil integral theory. If to restrict HK$_r$-integral by requirement that the indefinite HK$_r$-integral is {\em continuous}, then it is included even in the classical Denjoy--Khintchine integral. We provide a direct argument demonstrating that this inclusion is proper.

\end{abstract}
\maketitle
\section{Introduction}
The concept of derivation defined with $L_r$ integral means has been applied to build a counterpart of Perron integration theory by Louis Gordon in \cite{L Gordon,L GordonPhD}. Analogous approach to Riemann integration was introduced and investigated much later in a work by Musial and Sagher~\cite{MusialSagher2004}, related to the first author's PhD dissertation~\cite{MusialphD}. Despite a connection between these Perron and Riemann (Henstock--Kurzweil) type integrals has been established in \cite{MusialSagher2004,MusialphD} ($L^r$-Perron integrability implies HK$_r$-integrability), a true impetus the topic has received only very recently, with a surprising result (co-authored by the first author of the present note) that the HK$_r$-integral is strictly more general than the $L^r$-Perron integral, contradicting an intuition from classical theory; see also \cite{MST1,MST2,MT,MT2,TM}.

In the present note we locate Musial\,\&\,Sagher's concept of HK$_r$-inte\-gration within another counterpart of the Henstock--Kurzweil integral, the so-called approximate Henstock--Kurzweil integral (AH-integral). Having shown here that HK$_r$-integral is included into AH-integral (Corollary~\ref{ehbwb}), we obtain that it is covered by the integrals of a numerous family of approximately continuous integrals (see~\cite{sss}), in particular, it is included in the Kubota integral (see \cite{Kub,Lu1}). Moreover, if to restrict HK$_r$-integral by requirement that the indefinite HK$_r$-integral is {\em continuous}, then it is included even in the classical Denjoy--Khintchine (aka wide Denjoy or D-) integral. We provide a direct argument demonstrating that this inclusion is proper. In fact our construction shows that the HK$_r$-integral does not cover even the Khintchine integral which is more narrow than the D-integral.
\section{Terminology and notation}
\subsection{Riemann approach terminology}
We work in a fixed {\em segment} (compact interval) $[a,b]\subset\mathbb R$ (in some cases we shall use an alternative notation $\langle a,b\rangle$ for a subsegment of $\mathbb R$ meaning that any of the points $a$ and~$b$ can be on the left to the other one). A  {\em tagged interval} is a pair $(I,x)$ where $I\subset[a,b]$ is a segment and $x\in I$ (its {\em tag}). We say that a tagged interval $(I,x)$ is {\em tagged in} a set $E\subset[a,b]$ if $x\in E$, while it {\em is in} $E$ if $I\subset E$. A {\em gauge} is a strictly positive function $\delta$ on $[a,b]$ (or on a subset of $[a,b]$). We say that a tagged interval $(I,x)$ is {\em$\delta$-fine} if $I\subset(x-\delta(x),x+\delta(x))$.

A {\em division} is a finite collection $\pi$ of tagged intervals, where distinct elements $(I',x')$ and $(I'',x'')$ in $\pi$ have $I'$ and $I''$ non-overlapping, i.e., without inner points in common. A division is said to be {\em tagged in} a set/{\em$\delta$-fine} if all its elements have the respective property.
\subsection{$L^r$ related notions}
Throughout this paper we assume that ${r\ge1}$ and consider the respective Lebesgue spaces $L^{r}[a,b]$. $\mu$ denotes the Lebesgue measure on $[a,b]$. We recall the definition of the $L^{r}$-Henstock--Kurzweil integral given in \cite{MusialSagher2004} and some related notions from~\cite{L Gordon}.
\begin{definition}\label{def1}
A function $f\colon[a,b]\to\mathbb R$ is said to be {\em$L^r$-Henstock--Kurzweil integrable} ({\em HK$_r$-integrable}) if there exists a function $F\in L^{r}[a,b]$ such that for any $ \varepsilon >0$ there exists a gauge $\delta$ with the property that for any $\delta$-fine division $\pi=\{([c_i,d_i],x_i)\}_{i=1}^q$,
\[\sum_{i=1}^{q}\left(\frac{1}{d_i-c_i}\int_{c_i}^{d_i}\lvert F(y)-F(x_i)-f(x_i)(y-x_{i})\rvert^r\mkern1.5mu dy\right)^{\!1/r}\!<\,\varepsilon.\]
\end{definition}
By \cite[Theorem 5]{MusialSagher2004}, the function $F$ in Definition \ref{def1} is unique up to an additive constant, so, putting $F(a)=0$, we can consider the indefinite HK$_r$-integral
\[F(x)=(\textup{HK}_{r}) \int_{a}^{x}f,\qquad x\in(a,b].\]
It can be easily checked that the value of HK$_r$-integral does not depend on the values of the function $f$ on a set of measure zero (provided only finite values of $f$ are considered).
\begin{definition}\label{def2}
{\rm A function $F\in L^{r}[a,b]$ is said to be {\em$L^{r}$-continuous at} $x\in[a,b]$ if
\[\lim_{h\rightarrow0}\frac1h\int_{-h}^{h}\lvert F(x+t) -F( x)\rvert^r\mkern1.5mu dt=0.\]
If $F$ is $L^r$-continuous at all $x\in E$, we say that $F$ is {\em L$^r$-continuous on~$E$}}.
\end{definition}
\begin{definition}\label{111}
{\rm  A function $F\in L^{r}[ a,b] $ is said to be {\it$L^r$-differentiable} at~$x$, if there exists a real number $\alpha$ such that
\[\left( \frac{1}{h}\int_{-h}^{h}\vert F( x+t) -F(x) -\alpha t\vert ^{r}\mkern1.5mu dt\right)^{\!{1/r}}=o( h).\]
In this case we say that $\alpha$ is the {\it$L^r$-derivative} of $F$ at $x$ and write $F_r'(x)=\alpha$}.
\end{definition}
It was proved in \cite{MusialSagher2004} that if $F$ is the indefinite HK$_r$-integral of $f$ then $F$ is $L^{r}$-continuous on $[a,b]$, and $F_r'$ exists and is equal to $f$ a.e.\ on $[a,b]$. For $F\in L^r[a,b]$, $x\in[a,b]$, and a tagged interval $(I,x)$ we denote
$$\Delta_rF(I,x)=\left(\frac1{|I|}\int_I|F(y)-F(x)|^r\mkern1.5mu dy\right)^{\!1/r}.$$

Now let us recall the notion of $L^r$-variational measure  generated by a function  $F\in L^r[a,b]$; it has been defined in~\cite{MST1}.
\begin{definition}\label{Lrv}
For $F\in L^r[a,b]$, a set $E\subset [a,b]$, and a fixed gauge $\delta$ on~$E$, we define the {\em$(\delta,r)$-variation} of $F$ on $E$ by
\begin{equation}\label{var}
\operatorname{Var}\mkern1.5mu(E,F,\delta, r)=\sup\sum_{i=1}^q\Delta_rF(I_i,x_i)
\end{equation}
where the sup is taken over all $\delta$-fine divisions $\{(I_i,x_i)\}_{i=1}^q$ in $[a,b]$ that are tagged in~$E$. The {\em$L^r$-variational measure} of $E\subset [a,b]$ generated by $F\in L^r[a,b]$ is defined by
$${\rm V}^r_{F}(E)=\inf_{\delta}\operatorname{Var}\mkern1.5mu(E,F,\delta, r)$$
where the inf is taken over all gauges on $E$. A variational measure ${\rm V}^r_F$ is said to be {\em absolutely continuous} on~$E$ (with respect to the Lebesgue measure~$\mu$) if ${\rm V}^r_{F}(N)=0$ for every $N\subset E$ such that $\mu(N)=0$. Note that, in particular, this condition applied to $N=\{x\}$ says ${\Delta_rF(I,x)\to0}$ as $|I|\to0$, i.e., that $F$ is $L^r$-continuous at~$x$ (and so on~$E$).
\end{definition}
The following descriptive characterization of the indefinite HK$_r$-integral in terms of $L^r$-variational measure  was given in~\cite{MST1}:
\begin{theorem}\label{chHKr} A function $f\colon[a,b]\to\mathbb R$ is HK$_r$-integrable if and only if there exists a function $F$ on $[a,b]$ with absolutely continuous ${\rm V}^r_{F}$ and such that $F_r'=f$ a.e.; the function $x\mapsto F(x)-F(a)$ is the indefinite HK$_r$-integral of~$f$.
\end{theorem}
\subsection{Approximate Kurzweil--Henstock integral}
The HK$_r$-integral, as we are going to demonstrate, can be related to the so-called {\em approximate Kurzweil--Henstock integral} ({\em AH-integral}). We avoid here its Riemann-type definition and introduce it via a characterization in terms of corresponding variational measure (similar to Theorem~\ref{chHKr}), which is sufficient for our needs here. More information on this theory can be found in \cite[Chapter 16]{R Gordon}.
\begin{definition}
A measurable set $E\subset\mathbb R$ is said to have a {\em density point} at $x\in\mathbb R$, if
$$\frac{\mu(E\cap[x-h,x+h])}{2h}\to1,\quad h\searrow0.$$
Analogously {\em lefthand} and {\em righthand density points} are defined. A function $F\colon[a,b]\to\mathbb R$ is said to be {\em approximately continuous} at $x\in(a,b)$ if for each $\varepsilon>0$, the set $\{y\in[a,b]:|F(y)-F(x)|<\varepsilon\}$ has $0$ as a density point. Analogously unilateral modes of this notion are defined. The corresponding notion of approximate differentiability and approximate derivative $F'_\textup{ap}$ one defines in the same fashion.
\end{definition}
An {\em approximate full cover} (AFC) on $E\subset\mathbb R$ is a collection $\mathcal S=\{\mathcal S_x\}_{x\in E}$ of measurable sets $\mathcal S_x\ni x$ having a density point at~$x$, for all $x\in E$.

For an AFC $\mathcal S=\{\mathcal S_x\}_{x\in E}$, we say a tagged interval $([y,z],x)$ is $\mathcal S$-fine if $y,z\in\mathcal S_x$. Analogously we understand $\mathcal S$-fine divisions. The following definition is analogous to Definition~\ref{Lrv}. Let $F\colon\mathbb R\to\mathbb R$, $E\subset\mathbb R$, and $\mathcal S$ be an AFC on~$E$. Set
$$\operatorname{Var}_{\textup{ap}}(E,F,\mathcal S)=\sup\sum_{i=1}^q|F(z_i)-F(y_i)|,$$
where sup ranges over all $\mathcal S$-fine divisions $\{([y_i,z_i],x_i)\}_{i=1}^q$ tagged in $E$, and
$${\rm V}^{\textup{ap}}_{F}(E)=\inf_{\mathcal S}\operatorname{Var}_{\textup{ap}}(E,F,\mathcal S).$$
The latter is called the {\em approximate variational measure} of $E$ generated by $F$. The approximate Kurzweil--Henstock integral can be characterized in terms of ${\rm V}^{\textup{ap}}_{F}$ (like the HK$_r$-integral in Theorem~\ref{chHKr}).
\begin{definition}\label{chAH}
We say a function $f\colon[a,b]\to\mathbb R$ is AH-integrable if there exists a function $F$ on $[a,b]$ with absolutely continuous ${\rm V}^{\textup{ap}}_{F}$ such that the approximate derivative $F_{\textup{ap}}'=f$ a.e.; the function $x\mapsto F(x)-F(a)$ is the indefinite AH-integral of~$f$.
\end{definition}
Note that the following result is true: every $F$ with absolutely continuous ${\rm V}^{\textup{ap}}_{F}$ is almost everywhere approximately differentiable, so that {\em the class of indefinite AH-integrals coincides with that of functions $F$, $F(a)=0$, generating absolutely continuous ${\rm V}^{\textup{ap}}_{F}$}; see \cite{ene,ss}.
\subsection{Khintchine and Denjoy--Khintchine integrals}
\begin{definition}
An $F\colon[a,b]\to\mathbb R$ is said to be an {\em ACG-function} if $[a,b]=\bigcup_{n=1}^\infty D_n$, where $F\restriction D_n$ is absolutely continuous for all~$n$. If, moreover, all $D_n$ can be chosen closed, we say $F$ is an {\em$[$ACG$]$-function}.
\end{definition}
\begin{definition}
A function $f\colon[a,b]\to\mathbb R$ is said to be Khintchine/Denjoy--Khintchine integrable if there exists a continuous ACG-function $F\colon[a,b]\to\mathbb R$ such that $F'(x)=f(x)$ (resp.\ $F'_{\textup{ap}}(x)=f(x)$) at almost all $x\in[a,b]$. One then defines $\int_a^bf=F(b)-F(a)$.
\end{definition}
\section{Results}
\begin{lm}\label{oei}
Every function $F\colon[a,b]\to\mathbb R$, $F\in L^r$, with absolutely continuous ${\rm V}^r_{F}$ generates also absolutely continuous ${\rm V}^{\textup{ap}}_{F}$.
\end{lm}
\begin{proof}
Given $x\in[a,b]$, denote for brevity $$\omega_x(h)=\Delta_rF(\langle x,x+h\rangle,x),\quad\ h\ne0,\ \langle x,x+h\rangle\subset[a,b].$$
We prove that if $x\in[a,b)$, $h>0$, the set
\begin{equation}\label{sci}
S_x(h)=\{t\in\langle0,h\rangle:|F(x+t)-F(x)|\le\omega_x(h)\}
\end{equation}
has $0$ as a righthand density point. If this were not true, for some $\eta>0$ arbitrarily small $k>0$ with $|C_k|>\eta k$, where $$C_k=\{t\in[0,k]:|F(x+t)-F(x)|>\omega_x(h)\},$$ should be found. For such a $k>0$, integrating $|F(x+t)-F(x)|^r$ over $C_k$ produces $k\cdot\omega_x(k)^r\ge\eta k\cdot\omega_x(h)^r$, i.e., $\omega_x(k)^r\ge\eta\cdot\omega_x(h)^r$, which can't hold (due to $L^r$-con\-ti\-nuity of~$F$) for arbitrarily small $k<h$ as long as $F$ isn't a.e.\ constant in some righthand neighborhood of~$x$ (since only then one could have $\omega_x(h)=0$). Similarly we prove that for $x\in(a,b]$ and $h<0$ the set \eqref{sci} has $0$ as a lefthand density point.

\smallskip
Assume now ${\rm V}^r_F$ is absolutely continuous and consider a nullset $N\subset[a,b]$. Then, given $\varepsilon>0$, there is a gauge $\delta$ such that if a division $\{(\langle x_i,y_i\rangle,x_i)\}_{i=1}^k$ is $\delta$-fine, tagged in~$N$, then $\sum_{i=1}^k\omega_{x_i}(y_i-x_i)<\varepsilon$. Let the sets $S_x(h),S_x(-h)$ be defined at $x$ with respect to $h=\delta(x)$. We set $\mathcal S_x=x+(S_x(h)\cup S_x(-h))$, $x\in N$. From what we have proved above, $\mathcal S_x$ has $x$ as a density point, i.e., $\mathcal S=\{\mathcal S_x\}_{x\in N}$ is an AFC on~$N$. If $\{([y_j,z_j],x_j)\}_{j=1}^l$ is an $\mathcal S$-fine division tagged in $N$, from the definition of $S_x(h)$ and $S_x(-h)$, it follows
\begin{align*}
\sum_{j=1}^l|F(y_j)-F(z_j)|&\le\sum_{j=1}^l|F(y_j)-F(x_j)|+\sum_{j=1}^l|F(x_j)-F(z_j)|\\
&\le\sum_{j=1}^l\omega_{x_j}(x_j-y_j)+\sum_{j=1}^l\omega_{x_j}(z_j-x_j)<\varepsilon+\varepsilon.
\end{align*}
That is,  $\operatorname{Var}_{\textup{ap}}(N,F,\mathcal S)\le2\varepsilon$ and so, since $\varepsilon>0$ was arbitrary, ${{\rm V}^{\textup{ap}}_{F}(N)=0}$.
\end{proof}
A straightforward consequence of the above argument is \cite[Theorem~6]{MusialSagher2004}.
\begin{remark}\label{gbuegnioe}
If $F\colon[a,b]\to\mathbb R$, $F\in L^r$, is $L^r$-continuous at $x\in[a,b]$, then it is approximately continuous at~$x$.
\end{remark}
\begin{wn}\label{ehbwb}
Every HK$_r$-integrable function is AH-integrable and both integrals coincide.
\end{wn}
\begin{proof}
Let $f\colon[a,b]\to\mathbb R$ be HK$_r$-integrable, $F=\int f$. By Theorem~\ref{chHKr}, $F'_r=f$ almost everywhere in $[a,b]$. At every point where $F$ is $L^r$-differentiable, it is approximately differentiable with the same derivative (see \cite[Theorem~2]{L Gordon}), that is, $F'_{\textup{ap}}(x)=f(x)$ at almost every $x\in[a,b]$. By Theorem \ref{chHKr} and Lemma~\ref{oei}, ${\rm V}_F^\textup{ap}$ is absolutely continuous, so $f$ is AH-integrable with ${\int f=F}$.
\end{proof}
In consequence, in view of \cite[Theorem 2.1]{Lu1}, we have a slight improvement of \cite[Corollary~1]{MusialSagher2004}.
\begin{wn}\label{eluacg}
Each indefinite HK$_r$-integral is an $[$ACG$]$-function.
\end{wn}
We show next that the HK$_r$-integral does not cover the Denjoy--Khintchine integral. Namely, we prove the following
\begin{theorem}\label{main}
There exists a function which is Khintchine (and so Denjoy--Khintchine) integrable on $[0,1]$ but which is HK$_r$-integrable on $[0,1]$ for no~$r$.
\end{theorem}
Let us remark that, thanks to Corollary \ref{ehbwb},  the theorem can be deduced from the theory of AH-integration; see e.g.~\cite{Fu1}. However, the argument we provide here has the advantage of being direct.
\begin{proof}
Let $C\subset[0,1]$ be the classical Cantor ternary set (of measure zero) with contiguous intervals $u_n$ of rank $n=1,2,\dots$ having length $u_n=3^{-n}$.\footnote{Here and in what follows, in a slight abuse of notation, we identify intervals $u_n,v_n,r_n$ and their lengths, so that there is no distinction made between different intervals of the same kind and the same rank.} The set that is left after removing all contiguous intervals up to rank $n$ from $[0,1]$, is constituted by $2^n$ segments $r_n$ of length $3^{-n}$ which are called {\em segments of rank}~$n$. Note that each $u_n$ is an interval concentric with some segment~$r_{n-1}$ (we put $r_0=[0,1]$). So each segment~$r_n$, for each nonnegative integer~$l$, can be represented as the union of $2^l$ segments $r_{n+l}$ and $2^l-1$ intervals $u_{n+l}$, all of them being of length~$3^{-(n+l)}$.

We construct a continuous ACG-function $F$ on $[0,1]$, differentiable everywhere outside of~$C$, which is the indefinite Khintchine integral of its derivative  $f=F'$ existing a.e.  We then show that $f$ is not HK$_r$-integrable. For each interval~$u_n$, let $v_n\subset u_n$ be the segment concentric with $u_n$ and of length $u_n/2$. We put $F=0$ on $C$ and $F=1/n$ on every~$v_n$, and then extend $F$ smoothly over both intervals $u_n\setminus v_n$ (in all $u_n$) so that the resulting function $F\ge0$  is continuous on $[0,1]$ and differentiable everywhere outside of~$C$. So we have for each contiguous interval~$u_n$,
\begin{equation}\label{1}
\int_{u_n}\!F^r>\int_{v_n}\!F^r=\frac{v_n}{n^r}=\frac1{2n^r\cdot3^n}.
\end{equation}
We show that the $L^r$-variational measure ${\rm V}^r_{F}$ is not absolutely continuous. Fix arbitrary gauge $\delta$ on $C$ and for $m\in\mathbb N$ denote $C_m=\{x\in C:\delta(x)>1/m\}$, $C=\bigcup_{m=1}^{\infty}C_m$. By the Baire category theorem there exists $m_0$ such that $C_{m_0}$ is metrically dense in some nonempty portion of $C$ defined by an interval $(c,d)$, i.e., $C\cap(c,d)$. We can assume that $d-c<1/{m_0}$. There exists a segment $r_n$ with $r_n\subset(c,d)$. As we have already noted, for each $l$ there are $2^l-1$ intervals $u_{n+l}$ of rank $n+l$ lying within the considered segment~$r_n$. Denote them $(\alpha_k, \beta_k)$, $1\leq k \leq 2^l-1$. Choose for each of them a point $x_k\in C_{m_0}$ which belongs to the segment $r_{n+l}$ adjoining $(\alpha_k, \beta_k)$ on the left. Then $$\beta_k - x_k\leq \frac{2}{3^{n+l}} <\frac1{m_0}<\delta(x_k)$$ and so pairs $([x_k, \beta_k],x_k)$ form a $\delta$-fine division tagged in~$C$. Having in mind \eqref{1} we get
\begin{align*}
\Delta_rF([x_k,\beta_k],x_k)&=\left(\frac1{\beta_k - x_k}\int_{x_k}^{\beta_k}|F(y)-F(x_k)|^r\mkern1.4mu dy\right)^{\!1/r}\\
&>\left(\frac{3^{n+l}}{2}\cdot\frac1{2(n+l)^r\cdot3^{n+l}}\right)^{\!1/r}=\frac1{4^{1/r}(n+l)}.
\end{align*}
Summing up over all $k$ we obtain
\[\operatorname{Var}\mkern1.5mu(C,F,\delta, r)\geq\sum_{k=1}^{2^l-1}\Delta_rF([x_k, \beta_k],x_k)>\frac{2^l-1}{4^{1/r}(n+l)}.\]
As $n$ is fixed and  $l$ is arbitrary we obtain that $\operatorname{Var}\mkern1.6mu(C,F,\delta,r)=\infty$ for any gauge~$\delta$. Hence ${\rm V}^r_{F}(C)=\infty$ and so ${\rm V}^r_{F}$ is not absolutely continuous. This proves that $F$ is not the indefinite HK$_r$-integral for its derivative.

To see that $F'$ is not HK$_r$-integrable it is enough to apply Remark \ref{gbuegnioe} and Corollary \ref{eluacg}. If this were not true, $G=(\textup{HK}_r)\int F'$ would be an approximately continuous [ACG]-function with $F'=G'_{\textup{ap}}$ a.e.\ in $[a,b]$. Thus, from the monotonicity property of [ACG]-functions~\cite{Kub}, $F-G=\textup{const}$, and so $F=G$, a contradiction.
\end{proof}
\begin{remark}\label{e}
An example in the opposite direction, i.e., of an HK$_r$-integrand which is not a Denjoy--Khintchine integrand, can be provided via a discontinuous and everywhere $L^r$-differentiable function (which can be constructed like $F$ in the proof of Theorem~\ref{main}, with $v_n/u_n\to0$ suitably quickly and $F=1$ over each~$v_n$).
\end{remark}
\begin{remark}\label{ee}
Note that an HK$_r$-integrable function whose indefinite integral is {\em continuous} is necessarily Denjoy--Khintchine integrable, while not necessarily Khintchine integrable. An example here is more delicate and can be constructed similarly as a continuous ACG-function which is not a.e.\ differentiable \cite[p.\,224]{saks}: take a Cantor-like set $H\subset[0,1]$ of positive measure with contiguous intervals $I_1,I_2,\ldots$; set $F=|I_n|+\rho_n$ on a segment $J_n\subset I_n$, where $\rho_n$ is the maximal length of a subinterval of $[0,1]$ disjoint to all $I_1,\dots,I_n$, $F=0$ elsewhere. It can be shown as in \cite{saks} that $F$ is not differentiable at any point of $H$, while if $J_n$ are thin enough in $I_n$ it is $L^r$-differentiable everywhere in $H$. Then one can modify $F$ smoothly around endpoints of all $J_n$ so that the resulting function is continuous, while nondifferentiability and $L^r$-differentiability on $H$ are not affected.
\end{remark}

\end{document}